\documentclass[11pt]{amsart}
\usepackage{amssymb}
\usepackage{amsthm}
\usepackage{amsmath}
\usepackage[dvips]{epsfig}
\usepackage{graphicx}
\usepackage{hyperref}
\usepackage{color}
\usepackage{url}
\usepackage{tikz}
\usepackage{caption}
\usepackage{subcaption}
\usepackage{epstopdf}
\usepackage{mathtools}
\usepackage[arrow, matrix, curve]{xy}
\usepackage{marginnote}

\usetikzlibrary{decorations.markings}

\def\theequation {\thesection.\arabic{equation}}
\makeatletter\@addtoreset {equation}{section}\makeatother

\newtheorem{theorem}{Theorem}

\newtheorem{lemma}{Lemma}
\theoremstyle{remark}
\newtheorem{remark}{Remark}
\theoremstyle{definition}
\newtheorem{definition}{Definition}
\theoremstyle{corollary}
\newtheorem{corollary}{Corollary}

{\begin{trivlist} \item[]{\em Proof }}%
{\hspace*{\fill}$\rule{.3\baselineskip}{.35\baselineskip}$\end{trivlist}}

\newcommand{\Z}{\mathbb{Z}}
\renewcommand{\d}{{\rm d}}

\DeclareMathOperator{\sech}{sech}

\DeclareMathOperator{\Real}{Re}

\DeclareMathOperator{\C}{\mathbb{C}}

\DeclareMathOperator{\R}{\mathbb{R}}
\DeclareMathOperator{\dom}{\rm dom}
\DeclareMathOperator{\ran}{\rm ran}

\begin{document}

\title[Spectral instability of the peaked periodic wave]{\bf Spectral instability
of the peaked periodic wave \\ in the reduced Ostrovsky equations}

\author{Anna Geyer}
\address[A. Geyer]{Delft Institute of Applied Mathematics, Faculty Electrical Engineering, Mathematics and
Computer Science, Delft University of Technology, Mekelweg 4, 2628 CD Delft, The Netherlands}
\email{A.Geyer@tudelft.nl}

\author{Dmitry E. Pelinovsky}
\address[D. Pelinovsky]{Department of Mathematics and Statistics, McMaster University,
Hamilton, Ontario, Canada, L8S 4K1}
\email{dmpeli@math.mcmaster.ca}
\address[D. Pelinovsky]{Department of Applied Mathematics, Nizhny Novgorod State Technical University, 24 Minin street, 603950 Nizhny Novgorod, Russia}


\thanks{DEP acknowledges a financial support from the State task program in the sphere
of scientific activity of Ministry of Education and Science of the Russian Federation
(Task No. 5.5176.2017/8.9) and from the grant of President of Russian Federation
for the leading scientific schools (NSH-2685.2018.5).}

\begin{abstract}
We show that the peaked periodic traveling wave of the reduced Ostrovsky equations with quadratic
and cubic nonlinearity is spectrally unstable in the space of square integrable periodic functions
with zero mean and the same period. The main novelty of our result is that
the spectrum of a linearized operator at the peaked periodic wave completely covers a closed vertical strip of the complex plane.
In order to obtain this instability, we prove an abstract result on spectra of operators under compact perturbations.
This justifies the truncation of the linearized operator at the peaked periodic wave to its
differential part for which the spectrum is then computed explicitly.
\end{abstract}

\date{\today}
\maketitle

\section{Introduction}
\label{sec-intro}

The Ostrovsky equation with the quadratic nonlinearity was originally derived by L.A. Ostrovsky
\cite{Ostrov} to model small-amplitude long waves in a rotating fluid of finite depth.
The same approximation was extended to internal gravity waves
in which case the underlying equation includes the cubic nonlinearity
and is referred to as the modified Ostrovsky equation \cite{Grimshaw1,Grimshaw2,Ostrov-review}.
When the high-frequency dispersion is neglected, the reduced Ostrovsky equation can be written in the form
\begin{equation}
\label{eq-redOst}
    u_t + u u_x = \partial_x^{-1} u,
\end{equation}
whereas the reduced modified Ostrovsky equation takes the form
\begin{equation}
\label{eq-redmodOst}
    u_t + u^2 u_x = \partial_x^{-1} u.
\end{equation}
For both equations (\ref{eq-redOst}) and (\ref{eq-redmodOst}),
periodic waves of the normalized period $2 \pi$ are considered in the Sobolev space of
$2\pi$-periodic functions denoted by $H^s_{\rm per} \equiv H^s_{\rm per}(-\pi,\pi)$, for some $s \geq 0$.
The subspace of $H^s_{\rm per}$ for $2\pi$-periodic functions with zero mean is denoted by $\dot{H}^s_{\rm per}$.
The operator $\partial_x^{-1}: \dot{H}^s_{\rm per} \rightarrow \dot{H}^{s+1}_{\rm per}$ denotes the anti-derivative
with zero mean.

Local well-posedness of the Cauchy problem for the reduced Ostrovsky equations (\ref{eq-redOst}) and (\ref{eq-redmodOst})
can be shown in $\dot{H}^s_{\rm per}$ with $s > \frac{3}{2}$ \cite{LPS1,SSK10}.
For sufficiently large initial data, the local solutions break in finite time,
similar to the inviscid Burgers equation \cite{JG,GH,LPS1}.
For sufficiently small initial data, the local solutions are continued for all times \cite{GP}.

Traveling wave solutions of the reduced Ostrovsky equations are of the form $u(x,t)=U(x-ct)$,
where $z=x-ct$ is the traveling wave coordinate and $c$ is the wave speed. The wave profile
$U(z)$ satisfies the integral-differential equation in the form
\begin{equation}
\label{ODE}
\left\{ \begin{array}{l} \left[ c - U(z)^p \right] U'(z) + (\partial_z^{-1} U)(z) = 0, \quad
\mbox{\rm for every } \; z \in (-\pi,\pi) \;\; \mbox{\rm with } \; U(z) \neq c, \vspace{0.5em}\\
U(-\pi) = U(\pi), \quad \int_{-\pi}^{\pi} U(z) dz = 0, \end{array} \right.
\end{equation}
where $p = 1$ for (\ref{eq-redOst}) and $p = 2$ for (\ref{eq-redmodOst}).

Smooth solutions to the boundary-value problem (\ref{ODE}) exist for $c \in (1,c_*)$, where
$c_*$ is uniquely defined, see \cite{GP17} (and \cite{Bruell18} for a generalization).
For $c \in (1,c_*)$ smooth solutions satisfy $U(z) < c$ for every $z \in [-\pi,\pi]$
and the boundary-value problem (\ref{ODE}) can be equivalently rewritten
in the differential form
\begin{equation}
\label{ode-1}
\left\{ \begin{array}{l} \frac{d}{dz} \left[( c - U(z)^p)\frac{dU}{dz} \right]  + U(z) = 0, \quad
\mbox{\rm for every } \; z \in [-\pi,\pi], \\
U(-\pi) = U(\pi), \quad U'(-\pi) = U'(\pi). \end{array} \right.
\end{equation}
At $c = c_*$, solutions to the boundary-value problem (\ref{ODE}) are
peaked at the points $z = \pm \pi$, where $U(\pm \pi) = c_*$.
Uniqueness and Lipschitz continuity of the peaked solutions to the boundary-value problem (\ref{ODE}) were proven
in \cite{GeyPel18} for $p=1$ (see \cite{Bruell18,Mats1} for a generalization). We denote
this unique (up to translation) peaked solution by $U_*(z)$.

For $p = 1$, the peaked wave $U_*(z)$ exists at the wave speed $c_* = \tfrac{\pi^2}{9}$ and
is given by
\begin{equation}
\label{eq-peakedwave}
  U_*(z) = \frac{1}{18} (3 z^2 - \pi^2), \quad \text{ for } z \in [-\pi,\pi],
\end{equation}
periodically continued beyond $[-\pi,\pi]$. It was already obtained in the original paper \cite{Ostrov}.
For $p = 2$, the peaked wave $U_*(z)$ exists
at the wave speed $c_* = \tfrac{\pi^2}{8}$ and is given by
\begin{equation}
\label{eq-peakedwave-mod}
  U_*(z) = \frac{1}{\sqrt{2}} \left(|z| - \frac{\pi}{2} \right), \quad \text{ for } z \in [-\pi,\pi],
\end{equation}
periodically continued beyond $[-\pi,\pi]$, see \cite{Nik}.
In both cases, $U_* \in \dot{H}^s_{\rm per}$
for $s < \frac{3}{2}$ with a finite jump discontinuity of the first derivative
at $z = \pm \pi$ for (\ref{eq-peakedwave}) and at $z = 0,\pm \pi$ for (\ref{eq-peakedwave-mod}).

Smooth periodic waves of the quasi-linear differential equation in (\ref{ode-1}) can be
obtained equivalently from a semi-linear differential equation by
means of the following change of coordinates \cite{JP,Hakkaev1,Hakkaev2}:
\begin{equation}
\label{coor-transf}
U(z) = {\bf u}(\xi),\quad    z = \int_0^{\xi} \left[ c - {\bf u}(s)^p \right] d s.
\end{equation}
The smooth periodic waves with profile ${\bf u}$ satisfy the differential equation
\begin{equation}
\label{ode-2}
\frac{d^2 {\bf u}}{d \xi^2} + \left[ c - {\bf u}(\xi)^p \right] {\bf u}(\xi) = 0.
\end{equation}
Although all periodic solutions of differential equation (\ref{ode-2}) are smooth, the coordinate transformation
(\ref{coor-transf}) fails to be invertible if ${\bf u}(\xi) = c$ for some $\xi$. Singularities
in the coordinate transformation are related to the appearance of the peaked solutions
in the boundary-value problem (\ref{ODE}).

Spectral stability of smooth periodic waves with respect to perturbations of the same period
was proven both for (\ref{eq-redOst}) and (\ref{eq-redmodOst}) in \cite{GP17,Hakkaev2}.
The analysis of \cite{GP17} relies on the standard variational formulation of the periodic waves
as critical points of energy subject to fixed momentum. The analysis of \cite{Hakkaev2}
relies on the coordinate transformation (\ref{coor-transf}), which reduces
the spectral stability problem of the form $\partial_x L v = \lambda v$
with the self-adjoint operator
$L = c - U^p + \partial_z^{-2}$
to the spectral problem of the form $M {\bf v} = \lambda \partial_{\xi} {\bf v}$ with the self-adjoint operator
$M = c - {\bf u}^p + \partial_{\xi}^2$.
The spectral problem $M {\bf v} = \lambda \partial_{\xi} {\bf v}$ has been studied before in \cite{Stefanov}
(see also \cite{Kostenko} for a generalization). Orbital stability of smooth periodic waves with respect to perturbations
of any period multiple to the wave period was proven in \cite{JP} by using higher-order conserved quantities
of the reduced Ostrovsky equations (\ref{eq-redOst}) and (\ref{eq-redmodOst}).

The peaked periodic waves are, informally speaking, located at the boundary between global and breaking solutions
in the reduced Ostrovsky equations. If the initial data $u_0$ is smooth, it was shown
that global solutions of (\ref{eq-redOst}) exist if $m_0(x) := 1 - 3 u_0''(x) > 0$ for every $x$
and wave breaking occurs if $m_0(x)$ is sign-indefinite \cite{GH,GP}, whereas
global solutions of (\ref{eq-redmodOst}) exist if $m_0(x) := 1 - \sqrt{2} |u_0'(x)| > 0$ for every $x$
and wave breaking occurs if $m_0(x)$ is sign-indefinite \cite{JG}. Substituting $U_*$ instead of $u_0$
yields $m_0(x) = 0$ almost everywhere except at the peaks. Thus, it is natural to expect that
the peaked periodic waves are unstable in the time evolution of the reduced Ostrovsky equations.

In  \cite{GeyPel18} we proved that the unique peaked solution \eqref{eq-peakedwave} of the
reduced Ostrovsky equation (\ref{eq-redOst}) is linearly unstable
with respect to square integrable perturbations with zero mean and the same period.
This was done by obtaining sharp bounds on the exponential growth of the $L^2$ norm of the
perturbations in the linearized time-evolution problem $v_t = \partial_z L v$. No claims regarding
the spectral instability of the peaked periodic wave were made in \cite{GeyPel18}.
In \cite{Hakkaev2}, explicit solutions of the spectral stability problem $M {\bf v} = \lambda \partial_{\xi} {\bf v}$ 
were constructed, but since this construction violated the periodic boundary conditions on the perturbation term,
it did not provide an answer to the spectral stability question.

The main goal of this paper is to show that {\em the peaked periodic wave $U_*$ is spectrally unstable
with respect to square integrable perturbations with zero mean and the same period}.
We achieve this for both versions of the reduced Ostrovsky equations (\ref{eq-redOst}) and (\ref{eq-redmodOst})
with the peaked periodic waves $U_*$ given in (\ref{eq-peakedwave}) and (\ref{eq-peakedwave-mod}), respectively.
We discover {\em an unusual instability of the peaked periodic wave:} the spectrum of the linearized operator
$A = \partial_z L$ in the space of $2\pi$-periodic mean-zero functions   completely covers a closed vertical strip
of the complex plane, as depicted in Figure \ref{fig-spectrum} for the reduced Ostrovsky equation (\ref{eq-redOst}).
The right boundary of this vertical strip with ${\rm Re}(\lambda) = \frac{\pi}{6}$
coincides with the sharp growth rate of the exponentially growing perturbations obtained in \cite{GeyPel18}
for the peaked wave $U_*$ given by (\ref{eq-peakedwave}).
The vertical strip remains invariant when  the spectrum of $A$ is defined in the space of subharmonic and localized  perturbations, see Remark \ref{rem_FloquetBloch}.

\begin{figure}[h!]
\center
\includegraphics[scale=0.8]{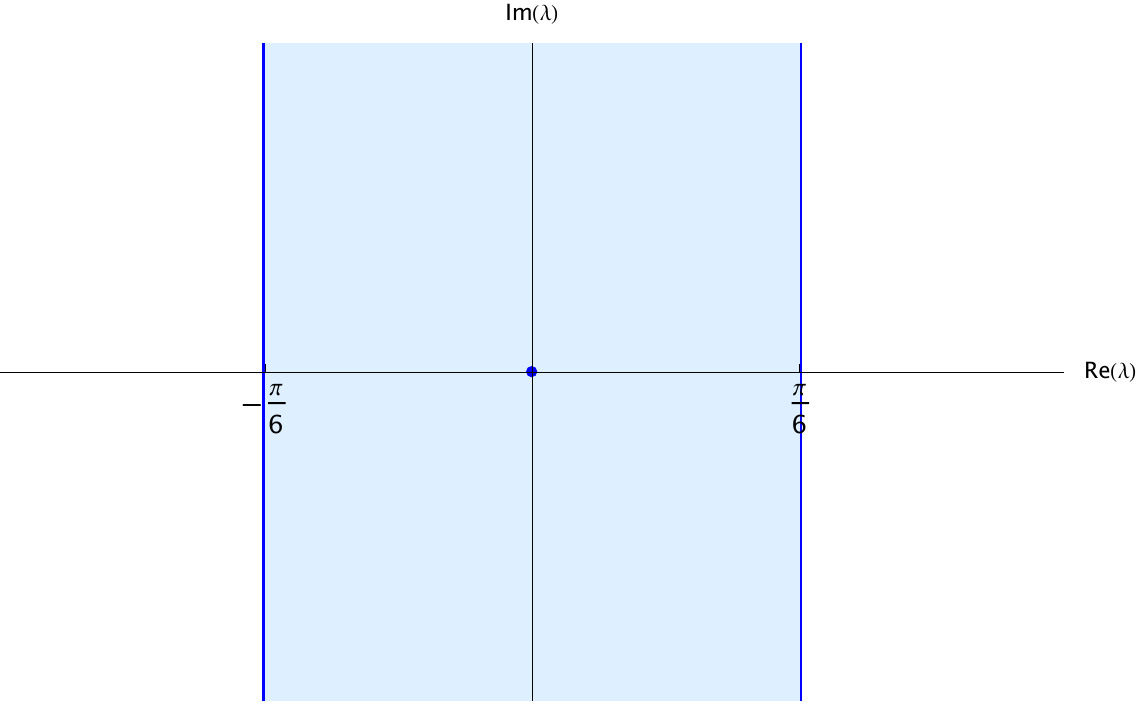}
\caption{The spectrum of the linearized operator $A$ at the peaked periodic wave $U_*$ given by (\ref{eq-peakedwave})
completely covers a closed vertical strip in the complex plane with zero being the only eigenvalue.
This shows that the peaked wave is spectrally unstable with respect to co-periodic perturbations.}
\label{fig-spectrum}
\end{figure}

Let us recall the following standard definition (see Definition 6.1.9 in \cite{BS18}).

\begin{definition}
Let $A$ be a linear operator on a Banach space $X$ with $\dom(A) \subset X$.
The complex plane $\mathbb{C}$ is decomposed into the following two sets:
\begin{enumerate}
    \item The resolvent set
    $$
    \rho(A) =  \left\{\lambda \in \C : \;\;
     \ker(A-\lambda I) = \{0\}, \;\; \ran(A-\lambda I) = X, \;\;  (A-\lambda I)^{-1}: X \to X \text{ is bounded} \right\}.
    $$
    \item The spectrum
    $$
    \sigma(A) =\C\setminus \rho(A),
    $$
    which is further decomposed into the following three disjoint sets:
    \begin{enumerate}
        \item the point spectrum
        $$\sigma_{\rm p}(A) = \{ \lambda \in \sigma(A) : \;\; \ker(A-\lambda I) \neq \{0\}\},$$
        \item the residual spectrum
        $$\sigma_{\rm r}(A) = \{ \lambda \in \sigma(A) : \;\; \ker(A-\lambda I) = \{0\}, \;\; \ran(A-\lambda I) \neq X \},$$
        \item the continuous spectrum
        $$\sigma_{\rm c}(A) = \{ \lambda \in \sigma(A) : \;\; \ker(A-\lambda I) = \{0\}, \;\; \ran(A-\lambda I) = X, \;\;
        (A-\lambda I)^{-1} : X \to X \text{ is unbounded}\}.$$
    \end{enumerate}\medskip
\end{enumerate}
\label{def-spectrum}
\end{definition}

In order to prove the spectral instability of the peaked periodic waves, we proceed as follows.
We first show that the point spectrum of the linearized operator $A$ consists of only the zero eigenvalue,
see Lemma \ref{lem-spectrum-A}. We then observe that $A$ is the sum of the linearization $A_0$ of
the quasi-linear part of the equation and a non-local term, which we may view as a compact perturbation $K$.
The truncated spectral problem for $A_0$ is then transformed to a problem on the line by a change coordinates
in Lemma \ref{lem-tech}. This facilitates the explicit computation of the spectrum of $A_0$ in
Lemmas \ref{lem-spectrumB0c} and \ref{lem-spectrumB0}. Finally, we  justify the truncation of
the linearized operator to its differential part by verifying the assumptions of the  following
abstract result, which is proven in the appendix.

\begin{theorem}
\label{thm-spectrumA0=A}
Let $A : {\rm dom}(A) \subset X \to X$ and $A_0 : {\rm dom}(A_0) \subset X \to X$ be
linear operators on Hilbert space $X$ with the same domain ${\rm dom}(A_0) = {\rm dom}(A)$ such
that $A - A_0 = K$ is a compact operator in $X$. Assume that the intersections
$\sigma_{\rm p}(A) \cap \rho(A_0)$ and $\sigma_{\rm p}(A_0) \cap \rho(A)$ are empty. Then,
$\sigma(A) = \sigma(A_0)$.
\end{theorem}

A similar instability with the spectrum lying in a vertical strip was discovered  in \cite{ShvLat03}  in the context of linearization around double periodic steady state solutions of the 2D Euler equations.

The proof of nonlinear instability of the peaked periodic waves is still open for the reduced Ostrovsky equations
(\ref{eq-redOst}) and (\ref{eq-redmodOst}). One of the main obstacles for nonlinear stability analysis is
the lack of well-posedness results for initial data in $\dot{H}^s_{\rm per}$ with $s < \frac{3}{2}$, which would
include the peaked periodic waves $U_*$ given by (\ref{eq-peakedwave}) and (\ref{eq-peakedwave-mod}). Another obstacle is the
discrepancy between the domain of the linearized operator $A = \partial_z L= \partial_z (c_* - U_*^p) + \partial_z^{-1}$ in $\dot{L}^2_{\rm per}$
and the Sobolev space $\dot{H}^1_{\rm per}$: while the former allows finite jumps of perturbations at the peaks,
the latter requires continuity of perturbations across the peaks,
see Remark \ref{rem-domain}. Because of a similar discrepancy, it is not clear if the Cauchy problem
for perturbations of the peaked periodic waves in the reduced Ostrovsky equation can be solved
in the domain of the iterated linearized operator $A^n$, $n \in \mathbb{N}$, which is again larger
than the Sobolev space $\dot{H}^n_{\rm per}$ of higher regularity, see Remark \ref{rem-Cauchy}.

The paper is organized as follows. The linearized operator $A$ is studied in Section 2 where
the main results for the peaked periodic waves $U_*$ given by (\ref{eq-peakedwave}) and (\ref{eq-peakedwave-mod})
are formulated. The proofs of the main results are contained in Section 3 and 4.
The appendix contains the proof of Theorem \ref{thm-spectrumA0=A}.

\section{Main results}

Linearizing (\ref{eq-redOst}) or (\ref{eq-redmodOst}) about the peaked traveling wave
$U_*(x-c_*t)$ with the perturbation $v(t,x-c_*t)$ yields an evolution problem of the form
\begin{equation}
\label{eq-linearization}
     v_t  = A v,
\end{equation}
where the operator $A : {\rm dom}(A) \subset  \dot L^2_{\rm per} \to \dot L^2_{\rm per}$ is defined by
\begin{equation}
\label{eq-L}
(Av)(z) \coloneqq \partial_z \left[ (c_* - U_*(z)^p) v(z)\right]  + \partial_z^{-1}v (z), \quad z \in (-\pi,\pi)
\end{equation}
with maximal domain
\begin{equation}
\label{dom-A}
{\rm dom}(A) = \left\{ v \in  \dot L^2_{\rm per} : \quad \partial_z  \left[ (c_* - U_*^p) v\right]  \in  \dot L^2_{\rm per} \right\},
\end{equation}
where either $p = 1$ for (\ref{eq-redOst}) or $p = 2$ for (\ref{eq-redmodOst}).

The linearized operator \eqref{eq-linearization} can be written as  $A = A_0 + K$, where the truncated operator
$A_0 : {\rm dom}(A_0) \subset \dot{L}^2_{\rm per} \to \dot L^2_{\rm per}$,
 is defined by
\begin{equation}
\label{eq-L1}
(A_0 v)(z) \coloneqq \partial_z \left[ (c_* - U_*(z)^p) v(z)\right], \quad z \in (-\pi,\pi)
\end{equation}
with the same domain ${\rm dom}(A_0) = {\rm dom}(A)$ and $K \coloneqq \partial_z^{-1}$ 
is a compact (Hilbert-Schmidt) operator in $\dot{L}^2_{\rm per}$ with spectrum 
$\sigma(K) = \{ i n^{-1}, \; n\in\Z\setminus\{0\}\}$.

By using Definition \ref{def-spectrum}, we introduce the following notion of spectral
stability for the traveling wave $U_*$.

\begin{definition}
\label{def-spectral-stability}
The traveling wave $U_*$ is said to be spectrally stable if $\sigma(A) \subset i \mathbb{R}$.
Otherwise, it is said to be spectrally unstable.
\end{definition}

The following two theorems present the main results of this paper.

\begin{theorem}
\label{thm-mainresult}
Consider the operator $A$ given by (\ref{eq-L}) on $\dot L^2_{\rm per}$
with $\dom(A)$ given by (\ref{dom-A}) for $p = 1$ and $U_*$ as in (\ref{eq-peakedwave}). Then,
\begin{equation}
\label{spectrum}
    \sigma(A) = \left\{ \lambda \in\C : \;\; -\frac{\pi}{6}\leq \Real(\lambda)\leq \frac{\pi}{6} \right\}.
\end{equation}
Consequently, the peaked wave $U_*$ is spectrally unstable in the reduced Ostrovsky equation
(\ref{eq-redOst}).
\end{theorem}

\begin{theorem}
\label{thm-mainresult-mod}
Consider the operator $A$ given by (\ref{eq-L}) on $\dot L^2_{\rm per}$
with $\dom(A)$ given by (\ref{dom-A}) for $p = 2$ and $U_*$ in (\ref{eq-peakedwave-mod}). Then,
\begin{equation}
\label{spectrum-mod}
    \sigma(A) = \left\{ \lambda \in\C : \;\; -\frac{\pi}{4}\leq \Real(\lambda)\leq \frac{\pi}{4} \right\}.
\end{equation}
Consequently, the peaked wave $U_*$ is spectrally unstable in the reduced modified Ostrovsky equation
(\ref{eq-redmodOst}).
\end{theorem}

\begin{remark}
\label{remark-zero}
One can find the explicit eigenvector for $0 \in \sigma(A)$ thanks to the translational symmetry
implying $A U_*' = 0$, where $U_*' \in {\rm dom}(A) \subset \dot{L}_{\rm per}^2$. Therefore, $0 \in \sigma_{\rm p}(A)$.
We show in Lemmas \ref{lem-spectrum-A} and \ref{lem-spectrum-A-mod} that $\sigma_{\rm p}(A) = \{ 0 \}$.
\end{remark}

\begin{remark}
We are not able to distinguish between residual and continuous spectrum in $\sigma(A) \backslash \{0\}$. This is because
we truncate the operator $A$ to an operator $A_0$ with the same domain and use the result of Theorem \ref{thm-spectrumA0=A}.
For the operator $A_0$ we prove in Lemmas \ref{lem-tech}, \ref{lem-spectrumB0c}, \ref{lem-spectrumB0}, and \ref{lem-tech-mod}  that
$\sigma_{\rm p}(A_0)$ is empty, $\sigma_{\rm r}(A_0)$ is the open vertical strip in (\ref{spectrum}) and (\ref{spectrum-mod}),
whereas $\sigma_{\rm c}(A_0)$ is the boundary of that vertical strip.
\end{remark}

\begin{remark}
The Sobolev space $\dot{H}^1_{\rm per}$ is continuously embedded into $\dom(A)$ in the sense that
there exists $C > 0$ such that for every $f \in \dot{H}^1_{\rm per}$, we have
$\partial_z  \left[ (c_* - U_*^p) f \right] \in  \dot L^2_{\rm per}$ with the bound
$$
\| \partial_z \left[ (c_* - U_*^p) f \right] \|_{L^2_{\rm per}} \leq C \| f \|_{H^1_{\rm per}}.
$$
However, $\dot{H}^1_{\rm per}$ is not equivalent to $\dom(A)$ because
piecewise continuous functions with finite jump discontinuities at the points $z$
where $c_* - U_*^p(z)$ vanishes belong to ${\rm dom}(A)$ but do not belong to $\dot{H}^1_{\rm per}$. For example,
the eigenvector $U_*' \in {\rm dom}(A)$ for $0 \in \sigma_{\rm p}(A)$
does not belong to $\dot{H}^1_{\rm per}$.
\label{rem-domain}
\end{remark}

\begin{remark}
One might ask whether looking at Sobolev spaces of higher regularity would result in a change of the spectrum
of the linearized operator at the peaked periodic waves. In order to answer this question,
let us introduce a hierarchy of the maximal domains of the iterated operator $A^n$ with $n \in \mathbb{N}$ by
\begin{equation*}
\dom(A^n) = \left\{ v \in  \dot L^2_{\rm per} : \quad \bigcap_{k=1}^n \left[ \partial_z  (c_* - U_*^p) \cdot \right]^k v  \in  \dot L^2_{\rm per} \right\}.
\end{equation*}
Then the operator $A_n : \dom(A^n) \subset \dom(A^{n-1}) \mapsto \dom(A^{n-1})$ for $n \geq 2$ has the same spectrum
as the operator $A : \dom(A) \subset \dot{L}^2_{\rm per} \mapsto \dot{L}^2_{\rm per}$
because the computations in Lemmas \ref{lem-spectrumB0c} and \ref{lem-spectrumB0}
are independent on $n$. The Sobolev space $\dot{H}^n_{\rm per}$
is continuously embedded into $\dom(A^n)$ but is not equivalent to $\dom(A^n)$, see Remark \ref{rem-domain}.
Consequently, it is not clear if the
Cauchy problem for periodic perturbations to the peaked periodic waves can be uniquely solved in any of the subspaces
of $\dot{L}^2_{\rm per}$ given by $\dom(A^n)$.
\label{rem-Cauchy}
\end{remark}

\section{Proof of Theorem \ref{thm-mainresult}}

For the peaked periodic wave $U_*$ in (\ref{eq-peakedwave}) in the case $p = 1$,
we write explicitly
\begin{equation}
\label{representation-1}
c_* - U_*(z) = \frac{1}{6} \left[ \pi^2 - z^2 \right], \quad z \in [-\pi,\pi].
\end{equation}
The eigenvector for $0 \in \sigma_{\rm p}(A)$ is given by
\begin{equation}
\label{eigenvector-1}
U_*'(z) = \frac{1}{3} z, \quad z \in (-\pi,\pi).
\end{equation}
The proof of Theorem \ref{thm-mainresult} can be divided into four steps.

\subsection*{Step 1: Point spectrum of $\boldsymbol A$.}

If $\lambda \in \sigma_p(A)$, then there exists $f \in {\rm dom}(A)$, $f\neq 0$, such that $A f = \lambda f$.
It follows from Remark \ref{remark-zero} that $0 \in \sigma_{\rm p}(A)$ with the eigenvector $U_*'$ in (\ref{eigenvector-1}).
The following result shows that no other eigenvalues of $\sigma_{\rm p}(A)$ exists.

\begin{lemma}
\label{lem-spectrum-A}
$\sigma_{\rm p}(A) = \{0\}$
\end{lemma}

\begin{proof}
First we note that if $f \in {\rm dom}(A)$, then $f \in H^1(-\pi,\pi)$ so that
$f \in C^0(-\pi,\pi)$ by Sobolev embedding. Bootstrapping arguments for $A f = \lambda f$
immediately yield that $f \in C^{\infty}(-\pi,\pi)$, hence the spectral problem
$A f = \lambda f$ for $f \in {\rm dom}(A)$ can be differentiated once in $z$
to yield the second-order differential equation
\begin{equation}
\label{ode-spectrum}
(\pi^2 - z^2) f''(z) - 4 z f'(z) + 4 f(z) = 6 \lambda f'(z), \quad z \in (-\pi,\pi).
\end{equation}
One solution is available in closed form: $f_1(z) = 2z + 3 \lambda$.
In order to obtain the second linearly independent solution,
we write $f_2(z) = (2z + 3 \lambda) g(z)$ and derive the following equation
for $g$:
\begin{equation}
\label{ode-spectrum-g}
(\pi^2 - z^2) (2z + 3 \lambda) g''(z) + 2 \left[ 2(\pi^2-z^2) - (2z + 3 \lambda)^2 \right] g'(z) = 0, \quad z \in (-\pi,\pi).
\end{equation}
This equation can be integrated once to obtain
\begin{equation}
\label{g-explicit}
g'(z) = \frac{g_0}{(\pi^2 - z^2)^2 (2 z + 3 \lambda)^2} \left( \frac{\pi + z}{\pi - z} \right)^{\frac{3\lambda}{\pi}}, \quad z \in (-\pi,\pi),
\end{equation}
where $g_0$ is a constant of integration. Computing the limits $z \to \pm \pi$ shows
that if $\pm 2 \pi + 3 \lambda \neq 0$, then
$$
g'(z) \sim \left\{ \begin{array}{l} (\pi - z)^{-\frac{3\lambda}{\pi} -2}, \quad z \to \pi, \\
(\pi + z)^{\frac{3\lambda}{\pi} -2}, \quad z \to -\pi. \end{array} \right.
$$
This sharp asymptotical behavior shows that $(\pi^2 - z^2) g'(z) \notin L^2(-\pi,\pi)$,  even if $g \in L^2(-\pi,\pi)$. Therefore, for every $\lambda \in \mathbb{C}$ with $\pm 2 \pi + 3 \lambda \neq 0$,
the second solution $f_2(z)$ does not belong to ${\rm dom}(A) \subset \dot{L}^2_{\rm per}$
because of the divergences as $z\to\pm \pi$. For $\pm 2 \pi + 3 \lambda = 0$, the explicit expression (\ref{g-explicit}) yields
$$
g'(z) = \frac{g_0}{4 (\pi^2 - z^2)^2 (\pi \pm z)^2}, \quad z \in (-\pi,\pi),
$$
which still implies that $f_2$ does not belong to ${\rm dom}(A) \subset \dot{L}^2_{\rm per}$.
Hence, for every $\lambda \in \mathbb{C}$, if $f \in {\rm dom}(A) \subset \dot{L}^2_{\rm per}$
is a solution to $Af = \lambda f$, then $f$ is proportional to
$f_1(z) = 2z + 3 \lambda$ only. The zero-mass constraint $\int_{-\pi}^{\pi} f_1(z) dz = 0$ required
for $f_1 \in \dot{L}^2_{\rm per}$ yields $\lambda = 0$, so that $f_1(z) = 2z = 6 U_*'(z)$ given by 
(\ref{eigenvector-1}). No other $\lambda \in \C$ such that a nonzero solution $f$ of
$A f = \lambda f$ belongs to ${\rm dom}(A) \subset \dot{L}^2_{\rm per}$ exists.
\end{proof}

\subsection*{Step 2: Truncation of $\boldsymbol A$.}

By using (\ref{representation-1}), $A_0$ in (\ref{eq-L1}) is rewritten in the explicit form
\begin{equation}
\label{eq-L1-explicit}
(A_0 v)(z) = \tfrac{1}{6} \partial_z \left[ (\pi^2-z^2) v(z) \right], \quad z \in (-\pi,\pi).
\end{equation}
Inserting the expression (\ref{representation-1}) in the transformation formula (\ref{coor-transf}) for $p=1$ yields
\begin{equation}
\label{z-xi}
    \frac{\d z}{\d \xi} = \tfrac{1}{6} (\pi^2- z^2),
\end{equation}
which we can solve to find that
\begin{equation}
\label{z-y}
    z = \pi \tanh\left (\frac{\pi\xi}{6}\right),
\end{equation}
where the constant of integration is defined without loss of generality from the condition that $z = 0$ at $\xi = 0$.
By using the explicit transformation formula (\ref{z-y}), we can rewrite the spectral problem $A_0 v = \lambda v$ in an equivalent but more convenient form.

\begin{lemma}
\label{lem-tech}
The spectral problem $A_0 v = \lambda v$ with $A_0 : {\rm dom}(A_0) \subset \dot{L}^2_{\rm per} \to \dot L^2_{\rm per}$
given by (\ref{eq-L1-explicit}) is equivalent to the spectral problem $B_0 w = \mu w$ with
\begin{equation}
\label{mu}
\mu = \frac{6}{\pi} \lambda,
\end{equation}
where $B_0 : {\rm dom}(B_0) \subset \tilde{L}^2(\R) \to \tilde{L}^2(\R)$ is the linear operator given by
\begin{equation}
\label{eq-tildeL1}
(B_0 w)(y) \coloneqq \partial_y w(y) - \tanh(y) w(y), \quad y \in \R,
\end{equation}
with maximal domain
\begin{equation}
\label{dom-B-0}
    {\rm dom}(B_0) = \left\{ w \in \tilde{L}^2(\R) : \; (\partial_y - \tanh y) w \in \tilde{L}^2(\R) \right\}= H^1(\R) \cap \tilde{L}^2(\R),
\end{equation}
where $\tilde{L}^2(\R)$ is the constrained $L^2$ space given by
\begin{equation}
\label{dot-L}
\tilde{L}^2(\R) \coloneqq \{ w \in L^2(\R) : \quad \langle w, \varphi \rangle = 0 \}
\end{equation}
with $\varphi(y) := {\rm sech}(y)$.
\end{lemma}

\begin{proof}
We first show that $v \in L^2(-\pi,\pi)$ if and only if $w\in L^2(\R)$. To this end, 
we use the substitution rule with \eqref{z-y}, set $y \coloneqq \frac{\pi\xi}{6}$ and write $v(z) = \cosh(y) w(y)$ to obtain that
\begin{align*}
    \int_{-\pi}^{\pi}v^2(z) \d z  = \pi \int_{-\infty}^{\infty}v^2(\pi \tanh y) \sech^2(y)\d y =\pi \int_{-\infty}^{\infty}w^2(y) \d y.
\end{align*}
Similarly, the zero-mean constraint in $\dot{L}^2_{\rm per}$ is transformed to
$$
0 = \int_{-\pi}^{\pi} v(z) dz = \pi \int_{-\infty}^{\infty} v(\pi \tanh y) \sech^2(y)\d y
= \pi \int_{-\infty}^{\infty} w(y) {\rm sech}(y) dy.
$$
Therefore, $v \in \dot{L}^2_{\rm per}$ if and only if $w \in \tilde{L}^2(\R)$.
Furthermore, we verify that
$$
\partial_z \left[ (\pi^2 - z^2) v \right] \in L^2(-\pi,\pi)
$$
if and only if
$$
\partial_y w - \tanh(y) w \in L^2(\R).
$$
Next we note that $B_0 w \in \tilde{L}^2(\mathbb{R})$ for every $w \in H^1(\mathbb{R})$, since
\begin{align}
\label{B0constraint}
\langle B_0 w, \varphi \rangle = \int_{\mathbb{R}} \left[ w'(y) - \tanh(y) w(y) \right] {\rm sech}(y) dy =
\int_{\mathbb{R}} \frac{d}{dy} \left[ w(y) {\rm sech}(y) \right] dy = 0.
\end{align}
This implies that the constraint $\langle B_0 w, \varphi \rangle = 0$ is identically satisfied for every $w \in H^1(\mathbb{R})$.
Moreover, if $w \in L^2(\R)$ and $[\partial_y - \tanh(y)] w \in L^2(\R)$, then $w \in H^1(\R)$.
The above arguments show that $B_0$ is closed in $\tilde{L}^2(\R)$ and
${\rm dom}(B_0) = H^1(\mathbb{R}) \cap \tilde{L}^2(\mathbb{R})$. Hence, the spectral problems for $A_0$ and $B_0$ are equivalent to each other and the spectral parameters $\lambda$ and $\mu$ are related by the transformation formula (\ref{mu}).
\end{proof}

\subsection*{Step 3: Spectrum of the truncated operator $\boldsymbol{A_0}$.}

In view of the equivalence of the spectral problems of $A_0$ and $B_0$ proven in Lemma \ref{lem-tech},  we proceed to study the spectrum of $B_0$ in $\tilde{L}^2(\R)$. The following two lemmas characterize the spectrum of $B_0$.

\begin{lemma}
\label{lem-spectrumB0c}
The point spectrum of $B_0$ is empty.
\end{lemma}

\begin{proof}
Let $\mu\in\C$ and $w \in \ker(B_0 -\mu  I)$, i.e.~$w$ satisfies the first-order differential equation
$$
\frac{dw}{dy} = \mu w(y) +\tanh(y) w(y).
$$
Solving this homogeneous equation yields
$$
w(y)=C \cosh(y) e^{\mu y}
$$
where $C$ is arbitrary. We have $w(y) \sim e^{(\mu \pm 1) y}$ as $y \to \pm \infty$
and hence the two exponential functions decay to zero as $y \to \pm \infty$ in two disjoint
sets of $\C$ for $\mu$. Hence, $w \in {\rm dom}(B_0) \subset \tilde{L}^2(\R)$ if and only if $C = 0$ for every $\mu \in \C$. We conclude that $w=0$, so $\sigma_p(B_0)=\emptyset$.
\end{proof}

\begin{lemma}
\label{lem-spectrumB0}
The residual spectrum of $B_0$ is
\begin{equation}
\label{B-0-res}
    \sigma_{\rm r}(B_0) = \left\{ \mu \in \C : \quad -1 < \Real(\mu) < 1 \right\},
\end{equation}
whereas the continuous spectrum of $B_0$ is
\begin{equation}
\label{B-0-cont}
    \sigma_{\rm c}(B_0) = \left\{ \mu \in \C : \quad \Real(\mu) = \pm 1 \right\}.
\end{equation}
\end{lemma}

\begin{proof}
Let $f\in \tilde{L}^2(\R)$, $\mu \in \C$, and consider the resolvent equation $(B_0-\mu I) w = f$, i.e.
\begin{equation}
\label{resolvent-eq}
\frac{dw}{dy} - \tanh(y) w(y) - \mu w(y) = f(y).
\end{equation}
Since the spectrum $\sigma(B_0)$ is invariant under translations along the imaginary axis, 
it suffices to study equation \eqref{resolvent-eq} for $\mu\in\R$, see also Theorem 3.13 in \cite{ChiLat99}. 
In what follows, we will study for which $\mu \in \R$  the resolvent equation (\ref{resolvent-eq})
has a solution $w$ in ${\rm dom}(B_0)$. Note that, if $\mu \neq 0$ and $w \in H^1(\R)$ 
is a solution to (\ref{resolvent-eq}), then the constraint $\langle f, \varphi \rangle = 0$ implies 
$\langle w, \varphi \rangle = 0$,  so that $w \in H^1(\mathbb{R})$ implies 
$w\in {\rm dom}(B_0) = H^1(\mathbb{R}) \cap \tilde{L}^2(\mathbb{R})$. On the other hand, if $\mu = 0$
and $w \in H^1(\R)$ is a solution to (\ref{resolvent-eq}), then
the constraint $\langle w, \varphi \rangle = 0$ is needed to ensure that $w \in {\rm dom}(B_0)$.

Solving the first-order inhomogeneous equation (\ref{resolvent-eq}) by variation of parameters yields
\begin{equation}
\label{sol-1}
w(y) = \cosh(y) e^{\mu y} \left[ C + \int_0^y e^{- \mu y'} {\rm sech}(y') f(y') \d y' \right],
\end{equation}
from which we infer that $w \in H^1_{\rm loc}(\mathbb{R})$. However, we also need to consider the behavior
of $w(y)$ as $y \to \pm \infty$ to ensure that $w\in  {\rm dom}(B_0)$.

Let us first show that the half line $I_+ := \{ \mu \in \R : \;\; \mu > 1\}$ belongs
to the resolvent set of $B_0$. Since $e^{(\mu + 1) y}$ diverges as $y \to +\infty$ for every $\mu \in I_+$, we define $C$ in \eqref{sol-1} by
\begin{equation}
\label{constant-C}
C \coloneqq -\int_0^{\infty} e^{-\mu y'} {\rm sech}(y') f(y') \d y',
\end{equation}
so that the unique solution (\ref{sol-1}) can be rewritten as
\begin{equation}
\label{sol-2}
w(y) = \int_{+\infty}^y e^{\mu (y-y')} \frac{\cosh(y)}{\cosh(y')} f(y') \d y'.
\end{equation}
The following two equivalent representations will be useful in  the estimates below:
\begin{eqnarray}
\label{repr-1}
\frac{\cosh(y)}{\cosh(y')} & = & \frac{1 + e^{2y}}{1 + e^{2y'}} e^{y'-y} \\
& = & \frac{1 + e^{-2y}}{1 + e^{-2y'}} e^{y-y'}.
\label{repr-2}
\end{eqnarray}
Let $f = f \chi_{\{y > 0\}} + f \chi_{\{y < 0\}}$, where $\chi_S$ is the characteristic function on the set $S \subset \R$, and define  $w_{\pm}$ by (\ref{sol-2}) with $f$ replaced by $f \chi_{\{\pm y > 0\}}$ so that $w = w_+ + w_-$.
Using (\ref{repr-1}) for $y < 0$ and (\ref{repr-2}) for $y > 0$, we obtain
$$
\text{for } y < 0 : \quad |w_+(y)| \leq  \int_0^{+\infty} e^{-(\mu - 1) (y'-y)} |f(y')| dy'
$$
and
$$
\text{for } y > 0 : \quad |w_+(y)| \leq 2 \int_y^{+\infty} e^{-(\mu   + 1)(y'-y)} |f(y')| dy'.
$$
By the Cauchy--Schwarz inequality for $y < 0$ and
by the generalized Young's inequality for $y > 0$, we obtain
\begin{align*}
    \|w_+ \|_{L^2(\R^-)} \leq \| e^{-(\mu-1)\cdot}\|_{L^2(\R^+)} \| f \|_{L^2(\R^+)} \| e^{(\mu-1) \cdot}\|_{L^2(\R^-)}
                        \leq \frac{1}{2( \mu -1)} \|f\|_{L^2(\R^+)}
\end{align*}
and
\begin{align*}
    \|w_+ \|_{L^2(\R^+)} \leq 2 \| e^{-(\mu+1)\cdot} \|_{L^1(\R^+)}  \|f\|_{L^2(\R^+)}
                                \leq \frac{2}{\mu+1} \|f\|_{L^2(\R^+)}.
\end{align*}
On the other hand, $w_-(y) = 0$ for $y > 0$ and
$$
y < 0 : \quad |w_-(y)| \leq 2 \int_y^0 e^{-(\mu - 1)(y'-y)} |f(y')| dy',
$$
where the representation (\ref{repr-1}) has been used. By the generalized Young's inequality, we obtain
\begin{align*}
    \|w_- \|_{L^2(\R)} \leq 2 \| e^{-(\mu-1)\cdot} \|_{L^1(\R^+)}  \|f\|_{L^2(\R^-)}
                                \leq \frac{2}{\mu-1} \|f\|_{L^2(\R^-)}.
\end{align*}
Putting these bounds together yields
\begin{equation}
\label{bound-on-inverse}
    \|w\|_{L^2(\R)} \leq C_{\mu}    \|f\|_{L^2(\R)},
\end{equation}
where the constant $C_{\mu} > 0$ depends on $\mu$ and is bounded for every $\mu >1$.
Thus, we have showed that $I_+  \in \rho(B_0)$. Similarly, one can show that
$I_- := \{ \mu \in \R : \;\; \mu < -1 \}$ also belongs to the resolvent set of
$B_0$ due to the same bound (\ref{bound-on-inverse}) for every $\mu \in I_-$.
Hence, $I_+ \cup I_- \subseteqq \rho(B_0)$. It remains to show that $[-1,1]\subseteqq \sigma(B_0)$.
More precisely, we show that $\mu \in \sigma_r(B_0)$ if  $\mu \in (-1,1)$ and $\mu \in \sigma_c(B_0)$
if  $\mu =\pm1$. We use again the explicit solution $w \in H^1_{\rm loc}(\R)$  given in (\ref{sol-1}).

If $\mu \in (-1,1)$, then the exponential functions $e^{(\mu + 1) y}$ and $e^{(\mu - 1) y}$ do not decay to zero as $y \to +\infty$ and $y \to -\infty$, respectively. Therefore, to ensure decay of $w(y)$ as $y \to \pm \infty$,
the constant $C$ in (\ref{sol-1}) would have to be defined twice
\begin{equation}
\label{constant-C-two-constraints}
C = -\int_0^{\infty} e^{-\mu y'} {\rm sech}(y') f(y') \d y' = \int_{-\infty}^0 e^{-\mu y'} {\rm sech}(y') f(y') dy'.
\end{equation}
This implies that $f \in \tilde{L}^2(\R)$ would have to satisfy an additional constraint
\begin{equation}
\label{constraint-f}
\int_{\mathbb{R}} e^{-\mu y'} {\rm sech}(y') f(y') dy' = 0,
\end{equation}
which is different from $\langle f, \varphi \rangle = 0$ if $\mu \neq 0$.
Fix $\mu \in \R$ such that $\mu \in (-1,1)$ and $\mu \neq 0$.
If $f \in \tilde{L}^2(\R)$ satisfies \eqref{constraint-f}, then there exists $w \in {\rm dom}(B_0)$ solution to \eqref{resolvent-eq}, since the previous analysis has shown that the solution $w$ given by
(\ref{sol-1}) with (\ref{constant-C-two-constraints}) decays to zero at infinity.
If $f \in \tilde{L}^2(\R)$ does not satisfy (\ref{constraint-f}), then
no such solution $w \in {\rm dom}(B_0)$ exists. Hence, there exist $f\in \dot L^2(\R)$ such
that for all $w\in {\rm dom}(B_0)$ we have $(B_0-\mu I)w \neq f$, i.e. ${\rm ran}(B_0-\mu I) \subsetneq \tilde{L}^2(\R)$.
This implies that this $\mu$ belongs to $\sigma_r(B_0)$.

In the special case $\mu = 0$, the constraint (\ref{constraint-f}) coincides
with $\langle f, \varphi \rangle = 0$. For $\mu = 0$ the unique solution (\ref{sol-1})
with $C$ as in (\ref{constant-C}) can be rewritten as
\begin{equation}
\label{sol-mu0}
w(y) = \int_{\infty}^y \frac{\cosh(y)}{\cosh(y')} f(y') \d y'.
\end{equation}
If $\langle f, \varphi \rangle = 0$, then the solution (\ref{sol-mu0}) belongs to $H^1(\mathbb{R})$. 
The constraint $\langle w, \varphi \rangle = 0$, however, is satisfied only under the additional constraint
\begin{equation}
\label{add-constraint-f}
\int_{\R} \int_{\infty}^y \sech(y') f(y') \d y' \d y = 0.
\end{equation}
Therefore, for $\mu = 0$, there exists no solution $w \in {\rm dom}(B_0)$ to
the resolvent equation (\ref{resolvent-eq}) unless $f \in \tilde{L}^2(\R)$ satisfies (\ref{add-constraint-f}).
This implies again that ${\rm ran}(B_0) \subsetneq \tilde{L}^2(\mathbb{R})$ and so $0 \in \sigma_r(B_0)$.
All together  we have established that $\sigma_{\rm r}(B_0)$ is given by (\ref{B-0-res}).

Finally, if $\mu = \pm 1$, one of the two exponential functions $e^{(\mu +  1) y}$
and $e^{(\mu - 1) y}$ in \eqref{sol-1} does not decay to zero both as $y \to +\infty$ and $y \to -\infty$.
Moreover, the improper integral in (\ref{sol-1}) does not converge for $f \in L^2(\R)$, $f \notin L^1(\R)$
because  $e^{\pm y'} \sech(y') \to 2$ as $y' \to \pm \infty$.
Therefore, the solution $w$ in (\ref{sol-1}) does not decay to zero and does not belong to ${\rm dom}(B_0)$
independently on the constraint on $C$ and hence $(B_0-\mu I)^{-1} : \tilde{L}^2(\R) \to \tilde{L}^2(\R)$ is unbounded.
We conclude that such $\mu$ belongs to $\sigma_{\rm c}(B_0)$ given by (\ref{B-0-cont}).
\end{proof}

\begin{corollary}
\label{cor-A0}
The spectrum of $A_0$ completely covers the closed vertical strip given by
\begin{equation}
\label{A-0-spectrum}
\sigma(A_0) = \left\{ \lambda \in\C : \;\; -\frac{\pi}{6}\leq \Real(\lambda)\leq \frac{\pi}{6} \right\}.
\end{equation}
\end{corollary}

\begin{proof}
The result follows from Lemmas \ref{lem-tech}, \ref{lem-spectrumB0c}, and \ref{lem-spectrumB0}.
\end{proof}

\subsection*{Step 4: Justification of the truncation.}

In this last step, we verify that the assumptions of the abstract Theorem \ref{thm-spectrumA0=A} hold for our operators.  Indeed,
by Lemmas \ref{lem-tech} and \ref{lem-spectrumB0c}, we have $\sigma_{\rm p}(A_0) = \sigma_{\rm p}(B_0) = \emptyset$.
Therefore, $\rho(A) \cap \sigma_{\rm p}(A_0) =\emptyset$. Moreover, Lemma \ref{lem-spectrum-A} states that $\sigma_{\rm p}(A) = \{ 0 \}$, 
hence Corollary \ref{cor-A0} implies that  $\rho(A_0) \cap \sigma_{\rm p}(A) =\emptyset$. Therefore, we may conclude from Theorem \ref{thm-spectrumA0=A}  that $\sigma(A) = \sigma(A_0)$, which together with (\ref{A-0-spectrum}) yields (\ref{spectrum}). This finishes the proof of Theorem \ref{thm-mainresult}.

\begin{remark}
\label{rem_FloquetBloch}
We can generalize  our instability result from co-periodic perturbations to \emph{subharmonic} and  \emph{localized perturbations} by analysing the Floquet-Bloch spectrum. In particular, we find that the spectrum of $A$ remains \emph{invariant} with respect to the Floquet exponent $k$ in the following decomposition:
\begin{equation*}
    v(z) = e^{i k z} p(z),
\end{equation*}
where $p(z+2\pi)=p(z)$ and $k\in[-\frac{1}{2},\frac{1}{2}]$. By setting  $z=\pi \tanh(y)$, $v(z)=\cosh(y) w(y)$ as in Lemma \ref{lem-tech} we rewrite the resolvent equation \eqref{resolvent-eq}  in the following form:
\begin{equation*}
    \frac{d q}{dy} - \tanh(y)  q(y) +i k \pi \sech^2(y) q(y) - \mu q(y) = g(y),
\end{equation*}
with $q(y)=e^{-ik\pi \tanh(y)}w(y)$ and $g(y)=e^{-ik\pi \tanh(y)}f(y)$. The general solution of this differential equation is obtained from \eqref{sol-1}
and given by
\begin{equation*}
q(y) = \cosh(y) e^{\mu y-ik\pi\tanh(y)} \left[ C + \int_0^y e^{- \mu y'+ik\pi\tanh(y')} {\rm sech}(y') g(y') \d y' \right].
\end{equation*}
Since $k$ is real, the analysis of this  solution is exactly the same as that of \eqref{sol-1} in the proof of Lemma \ref{lem-spectrumB0}. The estimates are independent of $k$, therefore the spectrum of the linearized operator $A$ remains the same when the co-periodic perturbations are replaced by subharmonic or localized perturbations.
\end{remark}

\begin{remark}
\label{remark-tilde}
If the constraint in (\ref{dot-L}) is dropped, one can define the differential operator
$\tilde{B}_0 : H^1(\R) \subset L^2(\R) \to L^2(\R)$, where $\tilde{B}_0$ has the same differential
expression as $B_0$ in (\ref{eq-tildeL1}). The proofs of Lemmas \ref{lem-spectrumB0c} and \ref{lem-spectrumB0}
are extended with little modifications to show that $\sigma_{\rm p}(\tilde{B}_0) = \emptyset$,
$\sigma_{\rm r}(\tilde{B}_0) = \sigma_{\rm r}(B_0)$, and $\sigma_{\rm c}(\tilde{B}_0) = \sigma_{\rm c}(B_0)$.
In addition, the same location of the spectrum of $\tilde{B}_0$ follows by Lemma 6.2.6 in \cite{BS18}.
Indeed, the adjoint operator $\tilde{B}_0^* : H^1(\R) \subset L^2(\R) \to L^2(\R)$ is defined by
\begin{equation*}
(\tilde{B}_0^* w)(y) \coloneqq -\partial_y w(y) - \tanh(y) w(y), \quad y \in \R
\end{equation*}
and the exact solution of the differential equation
\begin{equation*}
-\frac{dw}{dy} - \tanh(y) w(y) = \mu w(y)
\end{equation*}
is given by
\begin{equation*}
w(y) = C e^{-\mu y} {\rm sech}(y),
\end{equation*}
where $C$ is arbitrary. From the decay of exponential functions, we verify directly that
$\sigma_{\rm p}(\tilde{B}_0^*)$ is given by (\ref{B-0-res}) and
$\sigma_{\rm c}(\tilde{B}_0^*)$ is given by (\ref{B-0-cont}).
However, since $\sigma_{\rm p}(\tilde{B}_0) = \emptyset$, Lemma 6.2.6 in \cite{BS18}
implies that $\sigma_{\rm r}(\tilde{B}_0^*) = \emptyset$,
$\sigma_{\rm p}(\tilde{B}_0^*) = \sigma_{\rm r}(\tilde{B}_0)$, and $\sigma_{\rm c}(\tilde{B}_0^*) = \sigma_{\rm c}(\tilde{B}_0)$,
which is in agreement with the location of $\sigma_{\rm r}(\tilde{B}_0)$ and $\sigma_{\rm c}(\tilde{B}_0)$ obtained from direct computation.
\end{remark}

\section{Proof of Theorem \ref{thm-mainresult-mod}}
For the peaked periodic wave $U_*$ in (\ref{eq-peakedwave-mod}) in the case $p = 2$,
we write explicitly
\begin{equation}
\label{representation-2}
c_* - U_*^2(z) = \frac{1}{2} |z| \left( \pi - |z| \right), \quad z \in [-\pi,\pi].
\end{equation}
The eigenvector for $0 \in \sigma_{\rm p}(A)$ is given by
\begin{equation}
\label{eigenvector-2}
U_*'(z) = \frac{1}{\sqrt{2}} {\rm sign}(z), \quad z \in (-\pi,\pi).
\end{equation}
We follow the same four  steps as in the proof of Theorem \ref{thm-mainresult}. Note that now there
exist two peaks of the periodic wave (\ref{eq-peakedwave-mod}) on the $2\pi$-period:
one is located at $z = \pm \pi$ and the other one is located at $z = 0$. This modifies the proofs
of Lemmas \ref{lem-spectrum-A} and \ref{lem-tech} in Steps 1 and 2, whereas Steps 3 and 4 are exactly as in the case $p=1$.

\subsection*{Step 1: Point spectrum of $\boldsymbol A$.}

The following lemma is an adaptation of Lemma \ref{lem-spectrum-A} for the case $p = 2$.

\begin{lemma}
\label{lem-spectrum-A-mod}
$\sigma_{\rm p}(A) = \{0\}$
\end{lemma}

\begin{proof}
If $f \in {\rm dom}(A)$, then $f \in H^1(-\pi,0) \cap H^1(0,\pi)$ so that
$f \in C^0(-\pi,0) \cap C^0(0,\pi)$ by Sobolev embedding. Bootstrapping arguments for $A f = \lambda f$
immediately yield that $f \in C^{\infty}(-\pi,0) \cap C^{\infty}(0,\pi)$. Hence, the spectral problem
$A f = \lambda f$ for $f \in {\rm dom}(A)$ can be differentiated once in $z$ on $(-\pi,0)$ and $(0,\pi)$
to yield the second-order differential equation
\begin{equation}
\label{ode-spectrum-mod}
|z| (\pi - |z|) f''(z) + 2 {\rm sign}(z) (\pi - 2 |z|) f'(z) = 2 \lambda f'(z), \quad z \in (-\pi,0) \cup (0,\pi).
\end{equation}
Integrating (\ref{ode-spectrum-mod}) separately for $\pm z \in (0,\pi)$ yields
\begin{equation}
\label{g-explicit-mod}
f'(z) = \frac{g_{\pm}}{z^2 (\pi - |z|)^2} \left( \frac{z}{\pi - |z|} \right)^{\pm \frac{2\lambda}{\pi}}, \quad \pm z \in (0,\pi),
\end{equation}
where $g_{\pm}$ are constants of integration. Computing the limits $z \to 0$ and $z \to \pm \pi$
similarly to the proof of Lemma \ref{lem-spectrum-A} shows that $|z| (\pi - |z|) f'(z)$ belongs to $L^2(-\pi,0) \cap L^2(0,\pi)$ if and only if
$g_+ = g_- = 0$. In this case, $f(z) = f_{\pm}$ for $\pm z \in (0,\pi)$ with constant $f_{\pm}$
and the zero-mass constraint $\int_{-\pi}^{\pi} f(z) dz = 0$ required
for $f \in \dot{L}^2_{\rm per}$ yields $f_{\pm} = \pm f_0$ with only one scaling constant $f_0$.
Hence the only solution of $Af = \lambda f$ with $f \in {\rm dom}(A) \subset \dot{L}^2_{\rm per}$ is given by
$f(z) = f_0 \, {\rm sign}(z) = \sqrt{2} f_0 U_*'(z)$ given by (\ref{eigenvector-2}). 
Inspecting $A$ in (\ref{eq-L}) with $p = 2$ shows
that $(Af)(z)$ is even in $z$, whereas $\lambda f(z)$ is odd in $z$. Hence, $\lambda = 0$
is the only admissible value of $\lambda$ for this solution.
No other $\lambda \in \C$ exists  such that a nonzero solution $f$ of
$A f = \lambda f$ belongs to ${\rm dom}(A) \subset \dot{L}^2_{\rm per}$.
\end{proof}

\subsection*{Step 2: Truncation of $\boldsymbol A$.}
By using (\ref{representation-2}), $A_0$ in (\ref{eq-L1}) is rewritten in the explicit form
\begin{equation}
\label{eq-L1-explicit 2}
(A_0 v)(z) = \tfrac{1}{2} \partial_z \left[ |z| (\pi-|z|) v(z) \right], \quad z \in (-\pi,\pi).
\end{equation}
The explicit expression (\ref{representation-2}) in the transformation formula (\ref{coor-transf}) for $p=2$ yields
\begin{equation}
\label{z-xi-mod}
    \frac{\d z}{\d \xi} = \tfrac{1}{2} |z| (\pi - |z|).
\end{equation}
Both $z = \pm \pi$ and $z = 0$ are critical points
of (\ref{z-xi-mod}), so the interval $[-\pi,\pi]$
cannot be mapped bijectively to $\mathbb{R}$ as in the case $p=1$. However, we are able to
map the half-intervals $[-\pi,0]$ and $[0,\pi]$ between the two peaks
separately to $\mathbb{R}$. These maps are given explicitly as the solutions of (\ref{z-xi-mod}) by
\begin{equation}
\label{z-y-mod}
z = z_+(\xi) \coloneqq \frac{\pi e^{\frac{\pi \xi}{2}}}{1 + e^{\frac{\pi \xi}{2}}}
\end{equation}
 for $z\in [0,\pi]$, and
\begin{equation}
\label{z-y-mod2}
z = z_-(\xi) \coloneqq -\frac{\pi}{1+e^{\frac{\pi \xi}{2}}},
\end{equation}
 for $z\in [-\pi,0]$, where the constants of integration are defined without loss of generality from the conditions
$z_{\pm}(0) = \pm \frac{\pi}{2}$. The following is an adaptation of Lemma \ref{lem-tech}
when $p = 2$.

\begin{lemma}
\label{lem-tech-mod}
The spectral problem $A_0 v = \lambda v$ with $A_0 : {\rm dom}(A_0) \subset \dot{L}^2_{\rm per} \to \dot L^2_{\rm per}$
given by (\ref{eq-L1-explicit 2}) is equivalent to the spectral problem $B_0 w = \mu w$ with
\begin{equation}
\label{mu-mod}
\mu = \frac{4}{\pi} \lambda,
\end{equation}
where $B_0 : {\rm dom}(B_0) \subset \tilde{L}^2(\R) \to \tilde{L}^2(\R)$ is the same linear operator as
is given in (\ref{eq-tildeL1}) with the domain (\ref{dom-B-0}).
\end{lemma}

\begin{proof}
First, we consider the problem on the half-interval $[0,\pi]$.
By setting $y \coloneqq \frac{\pi}{4} \xi$ and $v(z_+) = \cosh(y) w_+(y)$, we obtain
by the substitution rule and using \eqref{z-y-mod} that
\begin{align*}
    \int_{0}^{\pi}v^2(z) \d z
    = \frac{\pi}{2} \int_{-\infty}^{\infty}v^2(z_+) \sech^2(y)\d y = \frac{\pi}{2} \int_{-\infty}^{\infty}w_+^2(y) \d y,
\end{align*}
hence $v \in L^2(0,\pi)$ if and only if $w_+\in L^2(\R)$. Similarly,
we verify that
$$
\partial_z \left[ z(\pi-z) v \right] \in L^2(0,\pi)
$$
if and only if
$$
\partial_y w_+ - \tanh(y) w_+ \in L^2(\R).
$$
Next, we consider the problem on the half-interval $[-\pi,0]$.
By setting $v(z_-) = \cosh(y) w_-(y)$ and using \eqref{z-y-mod2}, we obtain by the same computations
that $v \in L^2(-\pi,0)$ if and only if $w_-\in L^2(\R)$, whereas
$$
\partial_z \left[ z(\pi + z) v \right] \in L^2(-\pi,0)
$$
if and only if
$$
\partial_y w_- - \tanh(y) w_- \in L^2(\R).
$$
The zero-mean constraint in $\dot L^2_{\rm per}$ is transformed as follows:
\begin{align*}
        0 =\int_{-\pi}^{\pi} v(z) dz
            & = \frac{\pi}{2} \int_{\R} \left[ v(z_-)+v(z_+) \right] \sech^2(y) dy \\
            & = \frac{\pi}{2} \int_{\R} \left[ w_-(y)+w_+(y) \right] \sech(y) dy.
\end{align*}
Therefore, $v\in \dot L^2_{\rm per}$ if and only if $w \in \tilde{L}^2(\mathbb{R})$,
where $w \coloneqq w_+ + w_-$ and $\tilde{L}^2(\mathbb{R})$ is defined by (\ref{dot-L}).
In view of \eqref{B0constraint} we find that   $B_0 w\in \tilde{L}^2(\R)$ for $w=w_+ + w_- \in H^1(\R)$.
Considering the differential equation
$A_0 v = \lambda v$ on the half-intervals $[-\pi,0]$ and $[0,\pi]$,  we use the relations
$v(z_{\pm})=\cosh(y)w_{\pm}(y)$, the chain rule, and the transformation formula \eqref{mu-mod}
to obtain the equation $B_0 w_{\pm} = \mu w_{\pm}$, where the differential expression for $B_0$ is given by (\ref{eq-tildeL1}).
By the linear superposition principle, $w \in {\rm dom}(B_0) \subset \tilde{L}^2(\mathbb{R})$
defined by (\ref{dom-B-0}) satisfies the same equation $B_0 w = \mu w$ as $w_+$ and $w_-$.
Hence, the spectral problems for $A_0$ and $B_0$ are equivalent to each other and
the spectral parameters $\lambda$ and $\mu$ are related by the transformation formula (\ref{mu-mod}).
\end{proof}

\subsection*{Step 3: Spectrum of the truncated operator $\boldsymbol{A_0}$.}
Since the operator $B_0$ in Lemma \ref{lem-tech-mod} is identical with the one in Lemma \ref{lem-tech}, the results of Lemma \ref{lem-spectrumB0c} and \ref{lem-spectrumB0} apply directly to the case $p=2$ and give the following result.

\begin{corollary}
\label{cor-A0-mod}
The spectrum of $A_0$ completely covers the closed vertical strip given by
\begin{equation}
\label{A-0-spectrum-mod}
\sigma(A_0) = \left\{ \lambda \in\C : \;\; -\frac{\pi}{4}\leq \Real(\lambda)\leq \frac{\pi}{4} \right\}.
\end{equation}
\end{corollary}

\subsection*{Step 4: Justification of the truncation.}

In this last step, we verify that the assumptions of the abstract Theorem \ref{thm-spectrumA0=A} hold also in the case $p=2$.
Since $\sigma_{\rm p}(A_0)=\emptyset$,  $\rho(A) \cap \sigma_{\rm p}(A_0) =\emptyset$. Furthermore,
Lemma \ref{lem-spectrum-A-mod} states that $\sigma_{\rm p}(A) = \{ 0 \}$, hence Corollary \ref{cor-A0-mod} 
implies that $\rho(A_0) \cap \sigma_{\rm p}(A) =\emptyset$.
Therefore, we may conclude from Theorem \ref{thm-spectrumA0=A}  that $\sigma(A) = \sigma(A_0)$, which together with (\ref{A-0-spectrum-mod}) yields (\ref{spectrum-mod}). This finishes the proof of Theorem \ref{thm-mainresult-mod}.

\appendix

\section*{Appendix: Proof of Theorem \ref{thm-spectrumA0=A}}
\numberwithin{equation}{section}
\setcounter{equation}{0}
\renewcommand{\theequation}{A.\arabic{equation}}

Assume that $\lambda \in \sigma(A_0)$ but $\lambda \in \rho(A)$. Hence, for every $f \in {\rm dom}(A)$,
we can write
\begin{equation}
\label{aux-eq}
f = (A-\lambda I)^{-1} (K + A_0 - \lambda I) f,
\end{equation}
where $(A-\lambda I)^{-1} : X \to X$ is a bounded operator. The operator $(A-\lambda I)^{-1} K : X \to X$ is compact
as a composition of bounded and compact operators.
Therefore, the spectrum of $I - (A-\lambda I)^{-1} K$ in $X$ consists of eigenvalues accumulating at $1$.
Therefore, the Fredholm alternative holds: (i) either this operator is invertible for this $\lambda$ with a bounded inverse
or (ii) there exists $f_0 \in {\rm dom}(A)$, $f_0 \neq 0$ such that $f_0 = (A-\lambda I)^{-1} K f_0$.

In the case (i), we can rewrite (\ref{aux-eq}) for every $f \in {\rm dom}(A)$ in the form
\begin{equation}
\label{f-aux}
f = (I - (A-\lambda I)^{-1} K)^{-1} (A-\lambda I)^{-1} (A_0 - \lambda I) f,
\end{equation}
from which we obtain a contradiction against the assumption $\lambda \in \sigma(A_0)$.
Indeed, if $\lambda \in \sigma_{\rm p}(A_0)$, then there exists $f_0 \in {\rm dom}(A_0)$, $f_0 \neq 0$ such that
$(A_0 - \lambda I) f_0 = 0$, in which case equation (\ref{f-aux}) yields that $f_0 = 0$,  a contradiction. On the other hand,
if $\lambda \in \sigma_{\rm r}(A_0)$, then there exists $g_0 \in X$ such that $g_0 \notin {\rm ran}(A_0 - \lambda I)$.
This is in contradiction with (\ref{f-aux}) since for every $g_0 \in X$, there exists a unique $f_0 \in {\rm dom}(A)$
such that
$$
(A-\lambda I) (I - (A-\lambda I)^{-1} K) f_0 = g_0 = (A_0 - \lambda I) f_0.
$$
Finally, if $\lambda \in \sigma_{\rm c}(A_0)$, then for $f \in {\rm dom}(A_0)$ we
let $g := (A_0 - \lambda I) f \in X$ and obtain from (\ref{f-aux}) that
\begin{equation}
\label{f-aux-add}
\| f \|_X = \| (I - (A-\lambda I)^{-1} K)^{-1} (A-\lambda I)^{-1} g \|_X \leq C \| g \|_X,
\end{equation}
for some $C > 0$. Since $\lambda \in \sigma_{\rm c}(A_0)$, we have $\ran(A_0 - \lambda I) = X$ for this $\lambda$
and since $f \in {\rm dom}(A_0)$ is arbitrary, the bound (\ref{f-aux-add}) implies that for every $g \in X$,
$$
\| (A_0 - \lambda I)^{-1} g \|_X \leq C \| g \|_X,
$$
in contradiction with the assumption $\lambda \in \sigma_{\rm c}(A_0)$.

In the case (ii), there exists $f_0 \in {\rm dom}(A)$, $f_0\neq 0$, such that
$f_0 = (A-\lambda I)^{-1} K f_0$, and hence we can rewrite (\ref{aux-eq}) for this $f_0$ as
$$
(A-\lambda I)^{-1} (A_0 - \lambda I) f_0 = 0.
$$
Therefore, we have $(A_0 - \lambda I) f_0 = 0$, and hence $\lambda \in \sigma_{\rm p}(A_0)$,
in contradiction with the assumption that the intersection $\sigma_{\rm p}(A_0) \cap \rho(A)$ is empty.

Thus, if $\lambda \in \sigma(A_0)$, then $\lambda \in \sigma(A)$. Since $A_0 - A = - K$
and the previous argument does not depend on the sign of $K$, the reverse statement is true.
Hence, $\sigma(A) = \sigma(A_0)$.\qed


\end{document}